\documentclass{amsart}
\usepackage{a4wide}
\usepackage{mathrsfs}

\usepackage[T1]{fontenc}
\usepackage{microtype}
\usepackage{lmodern}
\usepackage[colorlinks=true,urlcolor=blue, citecolor=red,linkcolor=blue,linktocpage,pdfpagelabels, bookmarksnumbered,bookmarksopen]{hyperref}
\usepackage[hyperpageref]{backref}
\usepackage{amsthm} 
\usepackage{latexsym,amsmath,amssymb}

\usepackage{accents}
\usepackage{esint}
\usepackage[colorinlistoftodos,prependcaption,textsize=small]{todonotes}

\usepackage{soul}
\usepackage{mathtools} 
\usepackage{xparse} 
\usepackage[capitalize]{cleveref}

\usepackage[shortlabels]{enumitem}

\DeclareMathOperator{\Jac}{{\rm Jac}}

\newcommand{\AbsH}{|A|}

\title[Fractional Willmore-type energy]{A fractional Willmore-type energy functional -- subcritical observations}
\author[S. Blatt]{Simon Blatt}
\author[G. Giacomin]{Giovanni Giacomin}
\author[J. Scheuer]{Julian Scheuer}
\author[A. Schikorra]{Armin Schikorra}

\address[Simon Blatt]{Paris Lodron Universit\"at Salzburg, Hellbrunner Strasse 34, 5020 Salzburg, Austria}
\email{simon.blatt@sbg.ac.at}
\address[Giovanni Giacomin]
{Department of Mathematics and Statistics, University of Western Australia, 35 Stirling Highway, WA6009 Crawley,
Australia.} \email{giovanni.giacomin@research.uwa.edu.au}
\address[Armin Schikorra]{Department of Mathematics,
University of Pittsburgh,
301 Thackeray Hall,
Pittsburgh, PA 15260, USA}
\email{armin@pitt.edu}
\address[Julian Scheuer]{Goethe-Universität, Institut f\"ur Mathematik, Robert-Mayer-Str. 10, 60325
Frankfurt, Germany}
\email{scheuer@math.uni-frankfurt.de}

\newcommand{\N}{{\mathbb N}}

\renewcommand{\S}{{\mathbb S}}

\newtheorem{theorem}{Theorem}
\newtheorem{lemma}[theorem]{Lemma}
\newtheorem{corollary}[theorem]{Corollary}
\newtheorem{proposition}[theorem]{Proposition}

\theoremstyle{definition}

\theoremstyle{remark}
\newtheorem{remark}[theorem]{Remark}

\newcommand\diam{{\rm diam\,}}
\newcommand\dist{{\rm dist\,}}

\newcommand{\R}{\mathbb{R}}

\newcommand{\Z}{\mathbb{Z}}

\newcommand{\brac}[1]{\left (#1 \right )}
\newcommand{\abs}[1]{\left\lvert #1 \right \rvert}

\newcommand{\barint}{
\rule[.036in]{.12in}{.009in}\kern-.16in \displaystyle\int }

\newcommand{\barcal}{\text{$ \rule[.036in]{.11in}{.007in}\kern-.128in\int $}}

\def\mvint_#1{\mathchoice
          {\mathop{\vrule width 6pt height 3 pt depth -2.5pt
                  \kern -8pt \intop}\nolimits_{\kern -3pt #1}}%
          {\mathop{\vrule width 5pt height 3 pt depth -2.6pt
                  \kern -6pt \intop}\nolimits_{#1}}%
          {\mathop{\vrule width 5pt height 3 pt depth -2.6pt
                  \kern -6pt \intop}\nolimits_{#1}}%
          {\mathop{\vrule width 5pt height 3 pt depth -2.6pt
                  \kern -6pt \intop}\nolimits_{#1}}}

\numberwithin{theorem}{section} \numberwithin{equation}{section}

\newcommand{\aleq}{\lesssim}
\newcommand{\ageq}{\succsim}
\newcommand{\aeq}{\approx}

\usepackage{scalerel}[2014/03/10]
\usepackage[usestackEOL]{stackengine}
\def\avint{\,\ThisStyle{\ensurestackMath{%
			\stackinset{c}{.2\LMpt}{c}{.5\LMpt}{\SavedStyle-}{\SavedStyle\phantom{\int}}}%
		\setbox0=\hbox{$\SavedStyle\int\,$}\kern-\wd0}\int}

\let\latexchi\chi
\makeatletter
\renewcommand\chi{\@ifnextchar_\sub@chi\latexchi}
\newcommand{\sub@chi}[2]{
  \@ifnextchar^{\subsup@chi{#2}}{\latexchi^{}_{#2}}%
}
\newcommand{\subsup@chi}[3]{
  \latexchi_{#1}^{#3}%
}
\makeatother
\newcommand{\eps}{\varepsilon}

\begin{document}

\begin{abstract}
We investigate surfaces with bounded $L^p$-norm of the fractional mean curvature, a quantity we shall refer to as fractional Willmore-type functional. In the subcritical case and under convexity assumptions we show how this Willmore-functional controls local parametrization, and conclude as consequences lower Ahlfors-regularity, a weak Michael-Simon type inequality, and an application to stability.
\end{abstract}

\date{\today}
\maketitle
\tableofcontents

\section{Introduction}

Let $\Omega \subset \R^N$ be a bounded open set and consider (under appropriate smoothness assumptions) $\Sigma := \partial \Omega$ as a $d$-dimensional hypersurface, $N=d+1$.
For $s \in (0,1)$, Caffarelli-Roquejoffre-Savin initiated in \cite{CRS10} the study of the so-called fractional perimeter, and obtained as the first variation thereof the notion of fractional mean curvature $H_{\Sigma,s}$ given by the formula
\[
 H_{\Sigma,s} (x) := c_s \int_{\Sigma} \frac{\langle x-y, n(y) \rangle}{|x-y|^{d+1+s}}\, d\mathcal{H}^d(y),
\]
cf. \cite{CCC20}.
Here $n: \Sigma \to \S^{d} \subset \R^N$ is the outwards pointing unit normal on $\Sigma = \partial \Omega$ and $c_s$ is a constant.

This integral does not need to be considered in the principal value sense, observing that $\frac{x-y}{|x-y|}$ is approximately tangent to $\Sigma$ we see that $\lim_{y \to x} \langle \frac{x-y}{|x-y|}, n(y)\rangle =0$, and (under sufficient smoothness assumptions on $\partial \Omega$) we see that $H_{\Sigma,s}(x)$ is well-defined for any $x \in \Sigma$.

Also let us remark that by suitably choosing $c_s$ one can see that $\lim_{s \to 1} H_{\Sigma,s}(x) = H_{\Sigma}(x)$ where $H$ is the usual (scalar) mean curvature of $\Sigma$.

The definition of fractional perimeter and fractional mean curvature has spawned a new area of nonlocal geometric questions and problems, for e.g. fractional minimal surfaces (i.e. assuming $H_s \equiv 0$), fractional constant-mean-curvature surfaces (assuming $H_s \equiv const$). See, among many other works, e.g. \cite{CV11,CFSW18,DdPW18}.

This article is an invitation to consider a fractional version of the Willmore energy, i.e.
\[
 \mathcal{W}_{s,p}(\Sigma) := \int_{\Sigma} |H_{\Sigma,s}(x)|^p\, d\mathcal{H}^d(x).
\]
For $p = 2$ and $s \to 1$ (again, under a suitable choice of the normalizing constant $c_s$) we see that $\lim_{s \to 1} = \mathcal{W}(\Sigma)$, the classical Willmore energy that is of great interest in Geometric Analysis, involving the resolution of the Willmore conjecture not too long ago, \cite{MN14}, see also \cite{AFN16}.
Critical points of the Willmore energy are commonly referred to as Willmore surfaces, and Willmore surfaces clearly form a larger class than minimal surfaces.

We point out that a very related curvature energy to $\mathcal{W}_{s,p}$ has been studied for more than two decades in applied mathematics and topology: the class of integral tangent-point energies, that were first considered for modelling purposes, \cite{buck-orloff,GM99}. See \cite{SvdM12,SvdM13} for discussions under very weak assumptions of the underlying sets $\Sigma$, Sobolev-space characterizations \cite{BEnergyspaces13}, regularity for minimizers in the subcritical regime, \cite{BR2015}, and in the scaling-invariant case, \cite{BRSV21}.

The relation between fractional Willmore-energy and tangent-point energy becomes obvious once we write both formulas next to each other: the fractional Willmore energy is given by
\[
 \mathcal{W}_{s,p}(\Sigma) := |c_s|^p \int_{\Sigma} \abs{\int_{\Sigma} \frac{\langle x-y, n(y) \rangle}{|x-y|^{d+1+s}}\, d\mathcal{H}^d(y)}^p\, d\mathcal{H}^d(x),
\]
while the tangent-point energy is denoted as 
\[
 \mathcal{T}_{p,q}(\Sigma) := |c_s|^p \int_{\Sigma} \int_{\Sigma} \frac{\abs{\langle x-y, n(y) \rangle}^p}{|x-y|^{q-p}}\, d\mathcal{H}^d(y)\, d\mathcal{H}^d(x),
\]
for some parameters $q>p>0$.
The name \emph{tangent-point energy} derives from the fact that
\[
 R(x,y) := \frac{|x-y|^2}{|\langle n(y), x-y\rangle|}
\]
is the radius of the (smallest) sphere tangent to $y+ T_y \Sigma$ passing through $x$, so the tangent-point energy is closely related to an integrated norm of $1/R(x,y)$.

Tangent-point and related curvature energies\footnote{Those have been defined first for curves, often called knot energies, to capture topology of knots. We refer the interested reader e.g. to O'hara energies \cite{OH91,OH92,OH94}, a special case of which is the M\"obius energy \cite{FHW94}. See also \cite{BRS16,BRS19} and \cite{JunSurface1,JunSurface2}. Another popular class are the Menger curvature energies, \cite{SvdM13Menger}. } have been around for a long time in the topological-analytical community, and by the (formal) similarity of the corresponding nonlocal curvature quantities we hope that the two fields can positively influence each other.

Our analysis below will be restricted to two major simplifications: firstly we will \emph{de facto} mostly consider convex surfaces. 
Also, we focus on the subcritical case where $p > \frac{d}{s}$ -- with the hope that the arguments we present here are still useful for e.g. minimal and constant mean curvature. Observe that we have the scaling property
\begin{equation}\label{eq:willmorescaling}
 \mathcal{W}_{s,p}(\lambda \Sigma) = \lambda^{d-sp} \mathcal{W}_{s,p}(\Sigma) \quad \forall \lambda > 0.
\end{equation}
Thus, $p=\frac{d}{s}$ is the scaling invariant case, a case which requires more profound tools from Geometry and/or Harmonic Analysis. 

Our first result is the following uniform Lipschitz graph control.
\begin{theorem}\label{th:uniformbilipschitzparam}
Let $p>\frac{d}{s}$ and $s\in (0,1)$. Then for any $\Lambda > 0$ there exists $R > 0$ such that the following holds:

Let $\Sigma = \Sigma^d = \partial \Omega$ be the smooth boundary of a compact convex body in $\mathbb{R}^{d+1}$, such that
\[
 \mathcal{W}_{s,p}(\Sigma) \leq \Lambda.
\]
Then for any $x_0 \in \Sigma$ there exists a rotation $P \in SO(N)$ and a smooth map
\[
 f: B(0,R) \to \R
\]
with 
\[
 [\nabla f]_{C^{s-d/p}(B(0,R))} \leq C(\Lambda,s,p,d),
\]
so that $x_0+P(x,f(x))$ is a local parametrization of $\Sigma$, that is
\[
 \Phi (x) := x_0 + P(x,f(x))
\]
is a diffeomorphism as a map $B(0,R) \to \Phi(B(0,R))\subset \Sigma.$
\end{theorem}
Observe that it is by no means clear that the fractional Willmore-energy rules out distant strands to come close to $\Sigma$--unless we work with convex sets, where this self-repulsiveness holds similarly to tangent-point energies, \Cref{th:selfrepulsive}.

A corollary of \Cref{th:uniformbilipschitzparam} is uniform lower Ahlfor's regularity for (convex) sets with bounded fractional Willmore-energy in the subcritical regime.
\begin{corollary}[Uniform lower Ahlfor's regularity]\label{co:ahlfors}
Let $p> \frac{d}{s}$, $s \in (0,1)$ and $\Lambda > 0$. Then there exists a constant $c=c(p,d,s,\Lambda)$ and radius $R_0 = R_0(p,d,s,\Lambda)$ such that the following holds:

Let $\Sigma = \Sigma^d = \partial \Omega$ be the smooth boundary of a compact convex body in $\mathbb{R}^{d+1}$, such that 
\[
 \mathcal{W}_{s,p}(\Sigma) \leq \Lambda.
\]
Then for any $x \in \Sigma$
\[
 \inf_{R \in (0,R_0)} R^{-d} \mathcal{H}^d  (\Sigma \cap B(x,R)) > c
\]
\end{corollary}

Another immediate consequence is a sort of 
weak Michael-Simon type inequality:
\begin{corollary}[Weak Michael-Simon type inequality]\label{bt-5612}
Let $p> \frac{d}{s}$, $s \in (0,1)$, $\Lambda > 0$, $\alpha \in (0,1]$ and $q \in (1,\infty)$. Then there exists a constant $C=C(p,d,s,\alpha, q, \Lambda)$ such that the following holds.

Let $\Sigma = \Sigma^d = \partial \Omega$ be the smooth boundary of a compact convex body in $\mathbb{R}^{d+1}$, such that
\[
 \mathcal{W}_{s,p}(\Sigma) \leq \Lambda.
\]

Then for any $f \in W^{\alpha,q}(\Sigma)$ with $q \in (1,\frac{d}{\alpha})$ we have
\begin{equation}\label{eq:shittymichaelsimons}
\|f\|_{L^{\frac{dq}{d-\alpha q}}(\Sigma)} \leq C\brac{ [f]_{W^{\alpha,q}(\Sigma)} + \|f\|_{L^q(\Sigma)}}.
\end{equation}
If $\alpha-\frac{d}{q} = \beta$ for some $\beta \in (0,\alpha)$ then
\[
\|f\|_{L^\infty(\Sigma)} + [f]_{C^{\beta}(\Sigma)} \leq C \brac{[f]_{W^{\alpha,q}(\Sigma)} + \|f\|_{L^q(\Sigma)}}.
\]
\end{corollary}

A very similar inequality was obtained in Cabr\'e-Cozzi's \cite{CC19} who only need to assume a lower Ahlfors regularity \Cref{co:ahlfors}, but for all $R > 0$ (and for bounded sets that is clearly not possible). In our case, it follows from a simple covering argument, covering $\R^{d+1}$ with copies of cutoff-functions $\eta \in C_c^\infty(B(0,R))$, $\eta \equiv 1$ in $B(0,R/2)$, and using the local representation by graphs (and the classical Sobolev inequality).

Let us stress that a way stronger, real Michael-Simon inequality was recently obtained by Cabr\'e-Cozzi-Csat\'o \cite{CCC20} -- without the subcritical assumption $p>\frac{d}{s}$, however also using the convexity assumption.

Applying $f \equiv 1$ to \Cref{bt-5612}, equation \eqref{eq:shittymichaelsimons}, we obtain  the following lower bound on the area -- compare this with the scaling \eqref{eq:willmorescaling}. 
\begin{corollary}\label{co:areaundercontrol}
 Let $p> \frac{d}{s}$, $s \in (0,1)$ and $\Lambda > 0$. Then there exists a constant $c=c(p,d,s, \Lambda)$ such that the following holds.

Let $\Sigma = \Sigma^d = \partial \Omega$ be the smooth boundary of a compact convex body in $\mathbb{R}^{d+1}$, and assume
\[
 \mathcal{W}_{s,p}(\Sigma) \leq \Lambda
\]
Then 
\[
 \mathcal{H}^d(\Sigma) \geq c.
\]
\end{corollary}

The uniform parametrization is also useful when discussing flows and minimization problems -- observe however that the convexity assumption rules out surfaces $\Sigma$ of nontrivial genus, so that the only locally minimizing objects are most likely the sphere, cf. \cite{ACFGH03}.

We believe it would be very interesting to start an analysis of fractional Willmore-surfaces -- the scaling-invariant case -- e.g. for curves $\Sigma$ and $\mathcal{W}_{1/2,2}(\Sigma)$, eventually with the goal of developing as sophisticated arguments as for the classical Willmore-surfaces e.g. \cite{W93,Simon1985,BK03,KS04,Riv08}.

As a last point let us remark that without the convexity assumption everything discussed above still applies to a sort-of nonlocal bending energy:
Namely, the convexity assumption in \Cref{th:uniformbilipschitzparam} comes from the fact that we need to consider
\[
\AbsH_{\Sigma,s}(x) := c_s \int_{\Sigma} \frac{\abs{\langle x-y, n(y) \rangle}}{|x-y|^{d+1+s}}\, d\mathcal{H}^d(y).
\]
Observe that $\lim_{s \to 1^-}\AbsH_{\Sigma,s}(x)=|H(x)|$.
We can define then the nonlocal bending energy
\[
 \mathcal{B}_{s,p}(\Sigma) := \|\AbsH_{\Sigma,s}\|_{L^p(\Sigma)}^p.
\]
This sort of energy is even more similar to the tangent-point energy due to the absolute value inside the integral. 
If $\Sigma$ is the boundary of a \emph{convex} set, then 
\[
 \AbsH_{\Sigma,s}(x) = - H_{\Sigma,s}(x),
\]
and thus
\[
 \mathcal{W}_{s,p}(\Sigma) = \mathcal{B}_{s,p}(\Sigma).
\]
Indeed this property is the only reason why we need required convexity in the previous results. Namely we have:

\begin{proposition}
All the previous results are true for possibly \emph{non-convex} surfaces $\Sigma = \partial \Omega$, if we replace the fractional Willmore energy $\mathcal{W}_{s,p}$ by the bending energy $ \mathcal{B}_{s,p}$.
\end{proposition}

As an additional feature, we expect that $\mathcal{B}_{s,p}(\Sigma)$ is self-repulsive, meaning that distant strands cannot come close to each other. In other words, $\mathcal{B}_{s,p}(\Sigma)$ preserves embeddings. This is a feature known for subcritical (and to some extent also critical) tangent-point energies, \cite{strvdM,SvdM13}.

\begin{theorem}[Self-repulsiveness and chord-arc constant]\label{th:selfrepulsive}
Assume $p > \frac{d}{s}$ and $s\in (0,1)$. For every $\Lambda > 0$ there exists a constant $\Gamma =\Gamma(p,d,s,\Lambda)> 0$ such that if $\Sigma=\partial \Omega$ is a compact surface  as above (without any convexity assumptions) and
\[
 \mathcal{B}_{s,p}(\Sigma) \leq \Lambda \quad \text{and}\quad \mathcal{H}^d(\Sigma) \leq 1,
\]
then
\[
 |x-y| \leq d_{\Sigma}(x,y) \leq \Gamma |x-y|,
\]
where $d_{\Sigma}$ denotes the intrinsic distance within $\Sigma$.
\end{theorem}

The assumptions on $\mathcal{H}^d(\Sigma)$ is crucial. Recall that we are subcritical, so not scaling invariant, but rather
\[
 \mathcal{B}_{s,p}(\lambda \Sigma) = \lambda^{d-sp} \mathcal{B}_{s,p}(\Sigma), \quad \lambda >0.
\]
This implies that if $p>\frac{d}{s}$ for any family of finite energy surfaces $\mathcal{S}$ with the property $\lambda \mathcal{S} \subset \mathcal{S}$, we have
\[
 \inf_{\Sigma \in \mathcal{S}} \mathcal{B}_{s,p}(\Sigma) = 0.
\]
Note this does not contradict the uniform parametrization, \Cref{th:uniformbilipschitzparam}. So when we talk about minimizer, we have to bound the area from above.

\begin{theorem}[Existence of minimizers]\label{th:existencemin}
Let $p > \frac{d}{s}$ and $s\in(0,1)$. For every genus $g$ there exists an embedded $C^1$-surface $\Sigma$ minimizing the $\mathcal{B}_{s,p}$-energy among all embedded $C^1$-surfaces $\tilde{\Sigma}$ of the same genus with $\mathcal{H}^d(\tilde{\Sigma}) \leq 1$.
\end{theorem}
Observe again: if we assume a priori convexity, whence $\mathcal{B}_{s,p} = \mathcal{W}_{s,p}$, then the same result is true -- and thus we also find minimizers for the fractional Willmore energy $\mathcal{W}_{s,p}$, if $p > \frac{d}{s}$, among all \emph{convex} sets. However, there is no convex surfaces of nontrivial genus, so the above minimization problem is likely converging simply to the sphere (as of now, we do not know for a fact that the minimizer is the sphere).

The proof of \Cref{th:existencemin} almost immediately follows from the uniform bound on the H\"older constant of the derivative of the parametrization, \Cref{th:uniformbilipschitzparam}, and the non-collapsing guaranteed by \Cref{th:selfrepulsive}, combined with the arguments in \cite{Breuning12}.

We also remark that we are not aware of a commonly accepted way of extending the fractional mean curvature to higher co-dimension -- since there is no commonly accepted way of extending the perimeter to higher co-dimension, however see \cite{PPS22,fracsurfaceMihailaSeguin}. In contrast to this, it is easy to extend the functional $\mathcal{B}_{s,p}$ to higher-codimension manifolds $\Sigma^{d} \subset \R^N$,
\[
\mathcal{B}_{s,p}(\Sigma):= \int_{\Sigma} \brac{\int_{\Sigma} \frac{|\Pi^\perp(y) (x-y)|}{|x-y|^{d+s+1}} dy}^p.
\]
Here $\Pi^\perp(y)$ denotes the orthogonal projection from $\R^N$ onto $T_{y}^\perp \Sigma$. In this note we focus only on co-dimension $1$ surfaces, but our arguments easily extend to higher codimension $\mathcal{B}_{s,p}$ -- indeed with little extra work one obtains then minimizing manifolds in any ambient isotopy class (again in the subcritical case only). Again, the scaling invariant case $p = \frac{d}{s}$ is extremely challenging.

We conclude this introduction with several directions that we find interesting, are partially working on, but invite anyone interested to study

\begin{itemize}
 \item Classification of minimizers in subcritical and scaling-invariant case.
 \item The scaling-invariant case: existence of minimizers, regularity of minimizers.
 \item The fractional Willmore-flow, existence, uniqueness, and bubbling.
\end{itemize}

\subsection*{Acknowledgment}
Visits of A.S. to J. S., and G.G. to A.S. are gratefully acknowledged.
S.B., G.G., A.S. thank the mathematical research institute MATRIX in Australia where part of this research was performed. US-travel to this conference at MATRIX was funded by NSF DMS-2031696.

A.S. was funded NSF Career DMS-2044898. A.S. is a Alexander-von-Humboldt fellow. J.S. received funding through EPSRC, ``Stability for nonlocal curvature functionals'', No EP/W014807/2.

\section{Proof of Theorem~\ref{th:uniformbilipschitzparam}}

We believe that the following can be relatively easily generalized to higher co-dimension, and to surfaces $\Sigma$ with boundary or which do not bound a set. In particular it would apply to a convex or concave subset of $\Sigma$. For the (subcritical) tangent-point energies and probably many other energies the following argument is likely well-known, albeit possibly not written down.

\Cref{th:uniformbilipschitzparam} is a consequence of the following \Cref{th:uniformbilipschitzparam2}. Yet again we stress that if $\Sigma$ is the boundary of \emph{convex} sets then $\mathcal{B}_{s,p}(\Sigma) = \mathcal{W}_{s,p}(\Sigma)$, so the following theorem applies to the fractional Willmore energy for \emph{convex} sets $\Sigma$.

\begin{theorem}\label{th:uniformbilipschitzparam2}
Let $p>\frac{d}{s}$ and $s\in(0,1)$. Then for any $\Lambda > 0$ there exists $R =R(s,d,p,\Lambda) > 0$ such that the following holds:

Let $\Sigma$ be a closed smooth $d$-dimensional hypersurface, such that
\[
  \mathcal{B}_{s,p}(\Sigma) \equiv c_{s,p} \int_{\Sigma} \brac{\int_{\Sigma} \frac{\abs{\langle x-y, n(y) \rangle}}{|x-y|^{d+1+s}}\, d\mathcal{H}^d(y)}^p d \mathcal{H}^d(x)\leq \Lambda.
\]
Then for any $x_0 \in \Sigma$ there exists a rotation $P \in SO(N)$ and a smooth map
\[
 f: B(0,R) \to \R
\]
with
\[
 [\nabla f]_{C^{s-d/p}(B(0,R))} \leq C(\Lambda,s,p,d),
\]
 so that $x_0+P(x,f(x))$ is a local parametrization of $\Sigma$, that is
\[
 \Phi (x) := x_0 + P(x,f(x))
\]
is a diffeomorphism as a map $B(0,R) \to \Phi(B(0,R)\subset \Sigma$.
\end{theorem}
\begin{remark}
By scaling, one can easily make the constant $C(d,s,p,\Lambda)$ explicit in terms of $\Lambda$.
%
%
%
%
%
\end{remark}

Fix any point $x_0 \in \Sigma$. By translation we may assume that $x_0 = 0$. By a rotation $P$ we may assume that $\Sigma$ is horizontal at $0$, meaning that
\[
 T_0 \Sigma = \R^{d} \times \{0\}
\]
There is a small neighborhood $U \subset \R^{d+1}$ of $0$ such that
\[
 U \cap \Sigma
\]
can be represented as a graph, i.e. for some $\rho>0$ there is $f: B(0,\rho) \to \Sigma$ such that
\[
 \Phi(\cdot) := (\cdot,f(\cdot)) : B(0,\rho) \to \Sigma
\]
is a diffeomorphism onto its image. The idea is that if we have uniform control on $\nabla f$ then we can extend $f$ to a larger ball.
Indeed, this is a consequence of the implicit function theorem:
\begin{lemma}\label{la:ift}
Let $\Sigma$ be a smooth $d$-manifold without boundary in $\R^N$, $N=d+1$.
Fix some $\alpha > 0$. Assume $\rho > 0$ and that there exists a smooth $f: B(0,\rho) \to \R$ such that
\[
\Phi = (\cdot,f(\cdot)): B(0,\rho) \to \Sigma
\]
is a local parametrization. If $f \in C^1(B(0,\rho))$  with
\[
 \|\nabla f\|_{L^\infty(B(0,\rho))}\leq 1
\]
and uniformly continuous first derivatives, then there exists $\rho_2 > \rho$ and a smooth extension $f$ of $\tilde{f}$
\[
 \tilde{f}: B(0,\rho_2) \to \R
\]
such that
\[
\Phi = (\cdot,\tilde{f}(\cdot)): B(0,\rho_2) \to \Sigma
\]
is a local parametrization.
\end{lemma}
\begin{proof}
Since $f,\nabla f$ are uniformly continuous on the open set $B(0,\rho)$ there exists a $C^{1}$-extension of $f$ to $\overline{B(0,\rho)}$. By continuity we still have
\[
 \Phi(\cdot) := (\cdot,f(\cdot)): \overline{B(0,\rho)} \to \Sigma.
\]

Now fix any $x_0 \in \partial B(0,\rho)$. There exists a small tubular neighborhood $U \subset \R^{d+1}$ around $(x_0,f(x_0))$ for which the nearest point projection $\pi: U \to \Sigma$ exists. Set
\[
 d(x,t) := \langle (x,t) - \pi(x,t), n(\pi(x,t))\rangle, \quad \text{$(x,t) \in U$}.
\]
Here $n(\pi(x,t))$ denotes a choice of unit normal. Then $d$ is smooth, and we can ensure that
\[
 \forall (x,t) \in U:\quad (x,t) \in \Sigma\quad \Leftrightarrow\quad d(x,t) = 0.
\]
Indeed $(x,t) - \pi(x,t) = 0$ if $(x,t) \in \Sigma$, and if $(x,t) \not \in \Sigma$ then $(x,t) - \pi(x,t) \in (T_{\pi(x,t)} \Sigma)^\perp$.

If we can show that $\partial_t d(x_0,f(x_0)) \neq 0$, then by the implicit function theorem there exists a possibly smaller neighborhood $U$ (not relabeled) of $(x_0,f(x_0))$ and for some $\eps =\eps(x_0) > 0$ a map $\tilde{f}: B(x_0,\eps) \to \R$ such that
\[
 \forall (x,t) \in U:\quad (x,t) \in \Sigma\quad \Leftrightarrow\quad d(x,t) = 0 \quad \Leftrightarrow\quad t = \tilde{f}(x)
\]
Since $B(x_0,\eps) \cap B(0,\rho) \neq \emptyset$ we conclude that $\tilde{f}$ is an extension of $f$ to $B(x_0,\eps) \cup B(0,\rho)$. We can do this for any $x_0 \in \partial B(0,\rho)$ (which is compact). Take a finite sequence $(x_i,\eps(\rho_i))$ so that $V:=\bigcup B(x_i,\eps_i)$ covers $\partial B(0,\rho)$. Setting $U:=V \cup B(0,\rho)$ and $2\eps:=\inf\lbrace \abs{x-y}:\, x\in B(0,\rho)\,\mbox{and } y\in  \R^d\setminus U\rbrace$, we conclude that we have found an extension $\tilde{f}: B(0,\rho+\eps) \to \R$ with the desired properties.

So all we need to show is $\partial_t d(x_0,f(x_0)) \neq 0$. We observe that
\[
 \frac{d}{d\eps} \Big |_{\eps =0} d((x_0,f(x_0)) + \eps v) = 0 \quad \forall v \in T_{(x_0,f(x_0))} \Sigma.
\]
On the other hand if $v \in (T_{(x_0,f(x_0))} \Sigma)^\perp$ then $\pi((x_0,f(x_0))+ v) = (x_0,f(x_0))$ and thus
\[
 d((x_0,f(x_0))+\eps v) = \eps \underbrace{\langle v, n((x_0,f(x_0)))\rangle}_{\neq 0},
\]
and we conclude that
\[
\frac{d}{d\eps}\Big|_{\eps = 0} d((x_0,f(x_0)) + \eps v)) \neq 0 \quad \forall v \not \in T_{(x_0,f(x_0))} \Sigma.
\]
So, all we actually have to establish is that $(0,\ldots,0,1) \not \in T_{(x_0,f(x_0))} \Sigma$, then $\partial_t d(x,t) \neq 0$ and we are done.

But again, we observe that for each $x \in B(0,\rho)$ we have with the standard vectors $e_i\in \mathbb{R}^d$,
\[T_{(x,f(x))}\Sigma = {\rm span}\left\{(e_i,\partial_i f(x))\right\}_{1\leq i\leq d}.\]
By continuity this again holds for any $x \in \overline{B(0,\rho)}$. In particular $(0,\ldots,0,1) \not \in T_{(x_0,f(x_0))} \Sigma$ and we can conclude.
\end{proof}

The control that we need to apply \Cref{la:ift}, and which thus by a standard method of continuity concludes the proof of \Cref{th:uniformbilipschitzparam} and \Cref{th:uniformbilipschitzparam2} is the following:

%

\begin{lemma}\label{la:unifparam1}
Let $\Lambda > 0$, $s \in (0,1)$ and $p > \frac{d}{s}$. Then the following holds.

Assume that $\Sigma$ is a smooth $d$-manifold and that   $f: B(0,\rho) \to \R$ is a smooth map with $f(0) = 0$, $Df(0) = 0$, and set
\[
 \Phi(x) := (x,f(x). )
\]
Then whenever $\Phi: B(0,\rho) \to \R^{d+1}$ is a parametrization of $\Sigma$, we have
\[
\brac{\int_{B(0,\rho)} \brac{\int_{B(0,\rho)} \frac{\abs{f(x)-f(y)-Df(y) (x-y)}}{|x-y|^{d+1+s}}\, dy}^p\, dx}^{\frac{1}{p}} \aleq (1+ \|\nabla f\|_{L^\infty(B(0,\rho))})^\tau (\mathcal B_{s,p}(\Sigma))^{1/p}.
\]
where $\tau = d+2+s$.
\end{lemma}
\begin{proof}
Set $B=B(0,\rho)$.
Since $\Phi: B \to \R^{d+1}$ is a parametrization of $\Sigma$, identifying by an abuse of notation $n(y) = n(\Phi(y))$,
\[
\begin{split}
 \mathcal{B}_{s,p}(\Sigma) &=c_{s,p}\int_{\Sigma} \abs{\int_{\Sigma} \frac{\abs{\langle x-y, n(y) \rangle}}{|x-y|^{d+1+s}}\, d\mathcal{H}^d(y)}^p\, d\mathcal{H}^d(x)\\
 &\geq\int_{\Phi(B)} \abs{\int_{\Phi(B)} \frac{\abs{\langle x-y, n(y) \rangle}}{|x-y|^{d+1+s}}\, d\mathcal{H}^d(y)}^p\, d\mathcal{H}^d(x)\\
&=\int_{B} \abs{\int_{B} \frac{\abs{\langle \Phi(x)-\Phi(y), n(y) \rangle}}{|\Phi(x)-\Phi(y)|^{d+1+s}}\, \Jac(\Phi(y))d\mathcal{H}^d(y)}^p\, \Jac(\Phi(x))d\mathcal{H}^d(x).
\end{split}
 \]

Observe that
\[
 \Jac(\Phi) \ageq 1,
\]
and
\[
 |\Phi(x)-\Phi(y)| \leq \|D\Phi\|_{L^\infty(B)} |x-y| \aleq \left(1+\|\nabla f\|_{L^\infty(B)}\right)\, |x-y|.
\]
Thus we have
\begin{align}\begin{split}\nonumber
 \mathcal{B}_{s,p}(\Sigma)
 \ageq(1+\|\nabla f\|_{L^\infty(B)})^{-p(d+s+1)} \int_{B} \abs{\int_{B} \frac{\abs{\langle \Phi(x)-\Phi(y), n(y) \rangle}}{|x-y|^{d+1+s}}\, d\mathcal{H}^d(y)}^p\, d\mathcal{H}^d(x),
 \end{split}
 \end{align}
and hence
\begin{equation}\label{eq:sc:23241}
 \int_{B} \abs{\int_{B} \frac{\abs{\langle \Phi(x)-\Phi(y), n(y) \rangle}}{|x-y|^{d+1+s}}\, dy}^p\, dx \aleq \left(1+\|\nabla f\|_{L^\infty(B)}\right)^{p(d+s+1)}  \mathcal{B}_{s,p}(\Sigma).
 \end{equation}
Since $\partial_\alpha \Phi(y) \perp n(y)$ we have
\begin{equation}\label{eq:sc:23243}
\begin{split}
&\int_{B} \abs{\int_{B} \frac{\abs{\langle \Phi(x)-\Phi(y), n(y) \rangle}}{|x-y|^{d+1+s}}\, dy}^p\, dx
=\int_{B} \abs{\int_{B} \frac{\abs{\langle \Phi(x)-\Phi(y)-\partial_\alpha \Phi(y)(x-y)^\alpha, n(y) \rangle}}{|x-y|^{d+1+s}}\, dy}^p\, dx.
\end{split}
\end{equation}
On the other hand, since $\Phi(x) = (x,f(x))$, we have

\[ \brac{\Phi(x)-\Phi(y)-\partial_\alpha \Phi(y)(x-y)^\alpha}^i = (x-y)^i-(x-y)^i = 0, \quad i \in \{1,\ldots,d\}.
\]
Moreover
\begin{equation}\label{eq:normalgraph}
 n(y) = \pm \frac{1}{\sqrt{1+|\nabla f(y)|^2}} \left ( \begin{array}{c}
                                                 \nabla f(y)\\
                                                 -1
                                                \end{array} \right )
\end{equation}
and thus
\begin{equation}\label{eq:sc:23242}
\begin{split}
&\langle \Phi(x)-\Phi(y)-\partial_\alpha \Phi(y)(x-y)^\alpha, n(y) \rangle\\
=~&\mp \frac{1}{\sqrt{1+|\nabla f(y)|^2}} (\Phi(x)-\Phi(y)-\partial_\alpha \Phi(y)(x-y)^\alpha)^{d+1}\\
=~&\mp \frac{1}{\sqrt{1+|\nabla f(y)|^2}} (f(x)-f(y)-\partial_\alpha f(y)(x-y)^\alpha).
\end{split}
\end{equation}
Combining \eqref{eq:sc:23241}, \eqref{eq:sc:23243}, and \eqref{eq:sc:23242} we arrive at
\begin{align}\begin{split}\nonumber
 &\int_{B} \abs{\int_{B} \frac{\abs{ (f(x)-f(y)-\partial_\alpha f(y)(x-y)^\alpha)}}{\sqrt{1+|\nabla f(y)|^2}\,|x-y|^{d+1+s}}\, dy}^p\, dx
 \aleq \left(1+\|\nabla f\|_{L^\infty(B)}\right)^{p(d+s+1)}  \mathcal{B}_{s,p}(\Sigma)
\end{split}\end{align}
and consequently,
\begin{align}\begin{split}\nonumber
 &\brac{\int_{B} \abs{\int_{B} \frac{\abs{(f(x)-f(y)-\partial_\alpha f(y)(x-y)^\alpha)}}{|x-y|^{d+1+s}}\, dy}^p\, dx}^{\frac{1}{p}}
 \aleq \left(1+\|\nabla f\|_{L^\infty(B)}\right)^{d+2+s} \brac{\mathcal{B}_{s,p}(\Sigma)}^{\frac{1}{p}}.
\end{split}\end{align}

\end{proof}

\begin{lemma} \label{la:BetterControl}
Let $\Lambda>0$, $s\in (0,1)$ and $p>\frac ds$. Then there is a radius $R>0$ depending only on $\Lambda, d,s$ and $p$ such that for all $0< \rho < R$ the following holds:
Let $\Sigma$ be a closed smooth $d$-manifold, such that
\[
  \mathcal{B}_{s,p}(\Sigma) \equiv c_{s,p} \int_{\Sigma} \brac{\int_{\Sigma} \frac{\abs{\langle x-y, n(y) \rangle}}{|x-y|^{d+1+s}}\, d\mathcal{H}^d(y)}^p d \mathcal{H}^d(x)\leq \Lambda^p.
\]
 Let us furthermore assume that  for any $x_0 \in \Sigma$ there exists a rotation $P \in SO(N)$ and a smooth map
\[
 f: B(0,\rho) \to \R
\]
with $f(0)=0, \nabla f (0) =0$, and
\[
 \|\nabla f\|_{L^{\infty}(B(0,\rho))} \leq 1,
\]
so that $x_0+P(x,f(x))$ is a local parametrization of $\Sigma$, that is
\[
 \Phi (x) := x_0 + P(x,f(x))
\]
is a diffeomorphism as a map $B(0,\rho) \to \Sigma$.
Then,
$$\|\nabla f\|_{L^\infty(B(0,\rho))} \aleq  \Lambda \rho^{s-d/p} \leq \frac 1 2.$$
\end{lemma}

\begin{proof}
    Let us assume without loss of generality that $x_0=0$. Lemma \ref{la:unifparam1} tells us that
    
    \begin{align*}
&\brac{\int_{B(0,\rho)} \brac{\int_{B(0,\rho)} \frac{\abs{f(x)-f(y)-Df(y) (x-y)}}{|x-y|^{d+1+s}}\, dy}^p\, dx}^{\frac{1}{p}} 
\\ & \qquad \aleq (1+ \|\nabla f\|_{L^\infty(B(0,\rho))})^\sigma (\mathcal B_{s,p}(\Sigma))^{1/p} \\& \qquad \aleq \Lambda. 
\end{align*}
Let for simplicity $B = B(0,\rho)$. From Morrey-Sobolev embedding, \Cref{la:morreysob}, since $sp > d$
for $\sigma = s-\frac{d}{p}$ 
\[
 [Df]_{C^{\sigma}(B(0,\frac 34 \rho))} \aleq_{s,p} \brac{\int_{B}\brac{ \int_{B} \frac{\abs{f(x)-f(y)-Df(y) (x-y)}}{|x-y|^{d+1+s}}\, dy}^p\, dx}^{\frac{1}{p}}.
\]
Since by assumption $Df(0) = 0$, this implies 
\[
 \|\nabla f\|_{L^\infty(B(0, \tfrac 34 \rho))} \aleq \rho^{\sigma}\, \brac{\int_{B}\brac{ \int_{B} \frac{\abs{f(x)-f(y)-Df(y) (x-y)}}{|x-y|^{d+1+s}}\, dy}^p\, dx}^{\frac{1}{p}}
\]
and hence
\[
 \|\nabla f\|_{L^\infty(B(0, \tfrac 34 \rho))} \aleq \rho^{\sigma} \Lambda.
\]
Using this, we can already achieve that the gradient of $f$ is as small as we wish on the smaller ball $B(0,\tfrac 3 4 \rho)$. 

To get control over the gradient of $f$ on the complete ball $B(0,\rho)$, for a point $y \in B(0,\rho) \setminus B(0,\tfrac 34 \rho)$ we consider  $x_1 = \frac{3\rho}{4} \cdot \frac y {\|y\|}$ on $\partial B(0, \frac 3 4 \rho)$.  Then
$$ |\nabla f (x_1)| \aleq \rho^{\sigma} \Lambda.$$
Furthermore, we get using that the Lipschitz constant of $f$ on $B$ is bounded by one that  
$$
 d_\Sigma((x_1,f(x_1)), (y,f(y))) \leq \frac 1 4 \sqrt{2} \rho < \frac 34 \rho,
$$
where $d_\Sigma$ denotes the distance along the surface $\Sigma$.

Our assumption applied to the point $\tilde x_0= (x_1, f(x_1)) \in \Sigma$ now tells us that  there exists a rotation $\tilde P \in SO(N)$ and a smooth map
\[
 \tilde f: B(0,\rho) \to \R
\]
with $\tilde f(0)=0, \nabla \tilde f (0) =0$, and
\[
 \|\nabla \tilde f\|_{L^{\infty}(B(0,\rho))} \leq 1,
\]
so that $\tilde x_0+\tilde P(x,\tilde f(x))$ is a local parametrization of $\Sigma$, that is
\[
 \tilde \Phi (x) := \tilde x_0 + \tilde P(x,\tilde f(x))
\]
is a diffeomorphism onto its image as a map $B(0,\rho) \to \Sigma$.  As above we get that
$$
\|\nabla \tilde f \|_{L^\infty (B(0,\tfrac 3 4 \rho)} \aleq \rho^{\sigma} \Lambda
$$ and since $d_{\Sigma}(\tilde x_0 ,(y,f(y))) < \frac 34 \rho$ we get that 
$$
 (y,f(y)) \in \tilde \Phi(B(0,\tfrac 34 \rho))
$$
i.e. there is a point $\tilde y \in B(0, \frac 3 4 \rho)$ such that 
$$(y, f(y)) = \tilde x_0 + \tilde P (\tilde y, \tilde f (\tilde y)).$$
We will now deduce from this, that indeed 
$$
| \nabla f (y) | \aleq \rho^{{\sigma}} \Lambda \leq R^{\sigma} \Lambda \leq \frac 1 2
$$
whenever $R>0$ is sufficiently small.

To this end let us denote by 
$$
P^T_x: \R^{d+1} \rightarrow T_x \Sigma, \quad P^\bot_x: \R^{d+1} \rightarrow N_x \Sigma 
$$
the orthogonal projection onto the tangent space $T_x \Sigma$  and the normal space $N_x \Sigma$ along $\Sigma$ in the point $x \in \Sigma$ respectively.

We have
$$
 P_x^\bot(v) = \langle v, n \rangle n
$$
where $n=n(x)$ denotes a unit normal and hence can be considered as a Lipschitz function from the unit vectors onto the set of orthogonal projection equipped with the operator norm
.
If $\Sigma$ is the graph of a function $g$ around the point we consider, we obtain a local Gauß-map setting
$$
n = \frac {(\nabla g(x), - 1)} {\sqrt{|\nabla g(x)^2 +1 |}},
$$
i.e. the unit normal can be considered as a Lipschitz map of  the gradient under the condition that the gradient is bounded by $1$ as in our case.

Applying this first to $f$ and the points $0$ and $x_1$  and then to $\tilde f $ and the points $0$ and $\tilde y$ we get
\begin{align*}
| P^\bot_0 - P^\bot_{(y,f(y))} | & \leq | P^\bot_0 - P^\bot_{(x_1,f(x_1))} |+  |P^\bot_{(x_1,f(x_1))} - P^\bot_{(y,f(y))}|
\\ & \aleq |\nabla f(0) - \nabla f (x_1)| + |\nabla \tilde f(0) - \nabla \tilde f (\tilde y)| 
\\ & \aleq \rho^{\sigma} \Lambda.
\end{align*}
But from this we obtain 
\begin{align*}
|f(x)-f(y)| & = |P_0^\bot (x-y, f(x) - f(y))|  
\\ &\leq  |P_{(y,f(y))}^\bot (x-y, f(x) - f(y))| 
\\ & \quad + |P^\bot_{(y,f(y))}  - P^\bot _0 | |(x-y,f(x)-f(y))|.
\end{align*}

If we divide by $|x-y|$ let $x \rightarrow y$ and observe that  then 
\[P_{(y,f(y))}^\bot ( \tfrac 1 {|x-y} (x-y, f(x) - f(y))) \rightarrow 0,\] we get
$$
| \nabla f (y)| \aleq \Lambda \rho^{\sigma} \sqrt{1+ |\nabla f(y)|^2} \leq \Lambda \rho^{\sigma} (1+ |\nabla f(y)|).
$$
If   $R$ is small enough, this gives 
$$
 |\nabla f(y)| \leq C \Lambda \rho^{\sigma} \leq C\Lambda R^{\sigma} \leq \frac 12. 
$$
\end{proof}

\begin{proof}[Proof of Theorem \protect{\ref{th:uniformbilipschitzparam2}}]
We chose $\tilde \rho_0 > 0$  maximal with the property that for all $\rho < \tilde \rho_0$ the assumption of \Cref{la:BetterControl} are satisfied. Let us assume that $\tilde \rho_0 < R$ where again $R >0$ is the radius obtained from \Cref{la:BetterControl}.  Keeping in mind that $\Sigma$ is a compact $C^1$ - submanifold, the assumptions of Lemma \ref{la:BetterControl} are then also satisfied for $\rho = \tilde{\rho}_0$.

Then the previous lemma implies that around all points $x_0 \in \Sigma$  we can write $\Sigma$  up to translations and rotations as graphs of $C^{1,\sigma}$ - functions on $B(0, \tilde \rho_0)$ whose gradient is bounded by $\frac 12$. 
Since the submanifold $\Sigma$  is $C^1$ and compact, and hence the unit normals are uniformly continuous, we can combine Lemma \ref{la:ift} with a continuity argument, to prove that $\tilde \rho_0>0$ was not maximal. 
\end{proof}

\section{Proof of Michael-Simon type inequality, Corollary~\ref{bt-5612}}
In view of the equivalence of $\mathcal{B}_{s,q}(\Sigma)$ and $\mathcal{W}_{s,q}(\Sigma)$ for convex $\Sigma$, \Cref{bt-5612} is a consequence of the following
\begin{corollary}[Weak Michael-Simon type inequality]\label{bt-56122}
Let $p> \frac{d}{s}$, $s \in (0,1)$, $\Lambda > 0$, $\alpha \in (0,1]$ and $q \in (1,\infty)$. Then there exists a constant $C=C(p,d,s,\alpha, q, \Lambda)$ such that the following holds.

Let $\Sigma = \Sigma^d = \partial \Omega$ be the smooth boundary of a compact body in $\mathbb{R}^{d+1}$, such that
\[
 \mathcal{B}_{s,p}(\Sigma) \leq \Lambda.
\]

Then for any $f \in W^{\alpha,q}(\Sigma)$ with $q \in (1,\frac{d}{\alpha})$ we have
\begin{equation}
\|f\|_{L^{\frac{dq}{d-\alpha q}}(\Sigma)} \leq C\brac{ [f]_{W^{\alpha,q}(\Sigma)} + \|f\|_{L^q(\Sigma)}}.
\end{equation}
If $\alpha-\frac{d}{q} = \beta$ for some $\beta \in (0,\alpha)$ then
\[
\|f\|_{L^\infty(\Sigma)} + [f]_{C^{\beta}(\Sigma)} \leq C \brac{[f]_{W^{\alpha,q}(\Sigma)} + \|f\|_{L^q(\Sigma)}}.
\]
\end{corollary}
\begin{proof}[Proof of Corollary~\ref{bt-56122}: H\"older case]
Let $R\in (0,+\infty)$ be the $\Lambda$-dependent radius provided in \Cref{th:uniformbilipschitzparam2}. We can cover $\R^{d+1}$ with dyadic cubes $Q_i$ (with pairwise disjoint interior) and of sidelength $2^{-100} R$
\[
 \R^{d+1} = \bigcup_{i \in \Z} Q_i.
\]
Denote $B_j$ the ball with same center as $Q_j$ and radius $2^{-50} R$. Observe that there is a fixed dimensional constant $K$ such that any point in $\R^{d+1}$ is contained in at most $K$ balls $B_j$.

Subordinated to $B_j$ we choose a partition of unity, $\eta_j \in C_c^\infty(B_j)$, $\sum_{j=1} \eta_j(\cdot) = 1$ everywhere in $\R^{d+1}$.
Let  $J:=\left\lbrace j\in \Z| Q_j\cap\Sigma  \neq \emptyset \right\rbrace$.

We have
\[
 \Vert f \Vert_{L^{\infty}(\Sigma)} \leq \max_{j\in J} \Vert f \Vert_{L^{\infty}(\Sigma\cap B_j)}
\]
 and 
 \[
  [f ]_{C^{\beta}(\Sigma)} \aleq_R
 \Vert f \Vert_{L^{\infty}(\Sigma)} + \max_{j} [f ]_{C^{\beta}(\Sigma\cap B_j)}.
 \]

In view of \Cref{th:uniformbilipschitzparam2} we obtain for each $j\in J$ a parametrization
$$\Phi_j:B(0,R) \subset \R^d\to \Sigma $$ 
as provided in the statement of \Cref{th:uniformbilipschitzparam2}, where $\Phi_j(B(0,R))\supset B_j \cap \Sigma$, is a graph, and thus
\[
 |x-y| \aleq |\Phi_j(x)-\Phi_j(y)|.
\]
Thus (set $\beta := \alpha-\frac{d}{q}$).
\[
[\eta_j f ]_{C^{\beta}(\Sigma\cap B_j)} \leq [(\eta_j f) \circ \Phi_j]_{C^\beta(B(0,R))},
\]
and
\[
 \|(\eta_j f)\|_{L^\infty(\Sigma \cap B_j)} \leq \|(\eta_j f)\circ \Phi_j \|_{L^\infty(B(0,R))}.
\]
By Sobolev embedding,
\[
 \|(\eta_j f)\|_{L^\infty(\Sigma \cap B_j)} +
[\eta_j f ]_{C^{\beta}(\Sigma\cap B_j)}  \aleq_{R} [f\circ \Phi_j]_{W^{\alpha,q}(B(0,R))} +\|f\circ \Phi_j \|_{L^q(B(0,R))}.
\]
Since $\Phi_j$ are uniformly Lipschitz, we conclude
\[
\begin{split}
 \|f\|_{L^\infty(\Sigma)} +    [f ]_{C^{\beta}(\Sigma)}
 &\aleq \max_{j} [f]_{W^{\alpha,q}(\Sigma \cap B_j)} + \max_{j} \|f\|_{L^q(\Sigma \cap B_j)}\\
  &\aleq \max_{j} [f]_{W^{\alpha,q}(\Sigma)} + \max_{j} \|f\|_{L^q(\Sigma)}\\
  &= [f]_{W^{\alpha,q}(\Sigma)} + \|f\|_{L^q(\Sigma)}.
\end{split}
 \]

\end{proof}

\begin{proof}[Proof of Corollary~\ref{bt-56122}: Sobolev case]
Fix $\Sigma$ with $\mathcal{B}_{s,p}(\Sigma) \leq \Lambda$.

Our aim is to show
\[
 \brac{\int_{\Sigma} |f|^{q^\ast}}^{\frac{q}{q^\ast}} \aleq_{\Lambda} [f]_{W^{\alpha,q}(\Sigma)}^q + \|f\|_{L^q(\Sigma)}^q.
\]
with a constant independent of $\Sigma$.

%
%
%
%
%
%
%
%


With the same notation as in the previous proof,
in view of \Cref{th:uniformbilipschitzparam2} we obtain that for each for each $j\in J$ there exists a parametrization
$$\Phi_j:B(0,R)\to \Sigma $$
as provided in the statement of \Cref{th:uniformbilipschitzparam2}, where $\Phi_j(B(0,R))\supset B_j \cap \Sigma$.

Since $\frac{q}{q^\ast} < 1$, we have $\ell^1(\N) \subset \ell^{\frac{q^\ast}{q}}(\N) $ and we obtain
\[
\brac{\int_{\Sigma} |f|^{q^\ast}}^{\frac{q}{q^\ast}} \leq \brac{\sum_j \int_{\Sigma\cap B_j} |f|^{q^\ast}}^{\frac{q}{q^\ast}} \leq \sum_{j\in J}
 \|f\|_{L^{q^\ast}(\Sigma \cap B_j)}^{q}.
\]
With the uniform parametrization $\Phi_j: B(0,R) \subset \R^d \to \Sigma$ we have 
\[
 \|f\|_{L^{q^\ast}(\Sigma \cap B_j)} \aleq_{\Lambda} \|f\circ \Phi_j\|_{L^{q^\ast}(\Phi^{-1}_j (B_j))},
\]
By Sobolev embedding, with a constant depending on $\diam \Phi^{-1}_j(B_j)$ (which is controlled by $\Lambda$ due to the bi-Lipschitz control of $\Phi_j$)
\[
 \|f\|_{L^{q^\ast}(\Sigma \cap B_j)} \aleq_{\Lambda} [f\circ \Phi_j]_{W^{\alpha,q}(\Phi^{-1}_j(\Sigma \cap B_j))} + \|f\circ \Phi_j\|_{L^q(\Phi^{-1}_j(\Sigma \cap B_j))}.
\]
Reverting the parametrization $\Phi_j$ we then have 
\[
 \|f\|_{L^{q^\ast}(\Sigma \cap B_j)} \aleq_{\Lambda} [f]_{W^{\alpha,q}(\Sigma \cap B_j)} + \|f\|_{L^q(\Sigma \cap B_j)}
\]
and thus
\[
 \|f\|_{L^{q^\ast}(\Sigma \cap B_j)}^{q} \aleq_{\Lambda} [f]_{W^{\alpha,q}(\Sigma \cap B_j)}^{q} + \|f\|_{L^q(\Sigma \cap B_j)}^{q}.
\]
This implies 
\[
\brac{\int_{\Sigma} |f|^{q^\ast}}^{\frac{q}{q^\ast}} \aleq_{\Lambda} \sum_{j\in J}
 \brac{[f]_{W^{\alpha,q}(\Sigma \cap B_j)}^{q} + \|f\|_{L^q(\Sigma \cap B_j)}^{q}}.
\]
Since each point in $\R^{d+1}$ is covered by a maximal number of balls $B_j$, say $K$ (and $K$ depends only on the dimension) we conclude
\[
\sum_{j} \left([f]_{W^{\alpha,q}(\Sigma \cap B_j)}^{q} + \|f\|_{L^q(\Sigma \cap B_j)}^{q}\right) \leq K \brac{[f]_{W^{\alpha,q}(\Sigma)}^{q} + \|f\|_{L^q(\Sigma)}^{q}},
\]
and hence
\[
\brac{\int_{\Sigma} |f|^{q^\ast}}^{\frac{q}{q^\ast}} \aleq_{\Lambda} [f]_{W^{\alpha,q}(\Sigma)}^{q} + \|f\|_{L^q(\Sigma)}^{q}.
\]
This concludes the proof of Corollary~\ref{bt-56122}.
\end{proof}

\section{Stability for fractional curvature energies}

The energy $\mathcal{B}_{s,p}$ vanishes precisely on planes. However, in analogy with the classical traceless second fundamental form of a surface $\Sigma$, it would also be desirable to have an energy which characterizes the sphere. Inspired by the observation that  the linear map
\[T = A-\tfrac{1}{R}\,\mathrm{id} = D(n(x)-\tfrac{1}{R}x),\quad R>0\]
is zero precisely when $\Sigma$ is umbilic and must be a sphere of radius $R$, we define the quantity

\[\|A-\tfrac{1}{R}\,\mathrm{id}\|^{p}_{p,s,\Sigma} := c_{s,p}\int_{\Sigma}\int_{\Sigma}\frac{|(n(x)-\tfrac{1}{R}x)-(n(y)-\tfrac{1}{R}y)|^{p}}{|x-y|^{d+sp}}\,d\mathcal{H}^d(x)d\mathcal{H}^d(y).\]
This vanishes precisely on spheres of radius $R$. The aim of this section is to obtain a stability version of this rigidity statement, based on our weak Michael-Simon type inequality. Therefore it suffices to consider the normal component of the quantity $n(x) - \tfrac 1R x$, which is
\[1-\tfrac 1R u(x) := \langle n(x)-\tfrac 1R x,n(x)\rangle. \]
We define $R_{0}$ to be the mean-value of $u$, which is positive due to the divergence theorem. Then we obtain the following stability result:

\begin{theorem}
Let $\Sigma \subset \R^{d+1}$ be a closed hypersurface and suppose that {for some $s \in (0,1)$ and $p > \frac{d}{s}$ we have}
\[
 \mathcal{B}_{{s},{p}}(\Sigma) + {\diam(\Sigma)} \leq \Lambda.
\]
Let $\alpha q > d$. Then for any $\eps > 0$ there exists $\delta =\delta(\alpha,q,d,{\eps,\Lambda})> 0$ (otherwise independent of $\Sigma$) such that the following holds:
If
\[
 [u]_{W^{\alpha,q}(\Sigma)} < R_{0}\delta
\]
then
\[
 \dist(\Sigma,\S^d(R_{0})) < \eps,
\]
where 
$\mathbb{S}^{d}(R_{0})$ is a sphere of radius $R_{0}$ and where $\dist$ is the Hausdorff distance.
\end{theorem}
\begin{proof}
Apply our weak Michael-Simon type inequality, \Cref{bt-56122} to $R_0-u$,
then we have
\[
 [R_{0}-u]_{C^{\alpha-\frac{d}{q}}(\Sigma)} \aleq  [R_{0}-u]_{W^{\alpha,q}(\Sigma)} + \|R_{0}-u\|_{L^q(\Sigma)} \aleq_{\Lambda} R_{0}\delta,
\]
and due to the Poincar\'e-Sobolev inequality, 
\[
 \|R_{0}-u\|_{L^q(\Sigma)} \aleq \brac{\mathcal{H}^d(\Sigma)}^{-1} \diam(\Sigma)^{\frac{d}{q}+\alpha} [u]_{W^{\alpha,q}(\Sigma)}.
\]
Observe that we have control over $\mathcal{H}^d(\Sigma)$ in terms of $\Lambda$ by \Cref{co:areaundercontrol}.

We conclude that for sufficiently small $\delta$, the support function satisfies
\[u(x)=\langle x,n(x)\rangle >0\]
and hence $\Sigma$ is starshaped around the origin and it may be parametrized over a sphere by the function $r\colon\mathbb{S}^{d}\rightarrow \Sigma. $
Define
\[ v^{2} = 1+ r^{-2}|\nabla r|^{2}, \]
then 
\[\langle x,n(x)\rangle  = \frac{r}{v}\leq r\]
with equality precisely when $\nabla r = 0$. Hence
\[\min_{\mathbb{S}^{d}}\langle x,n(x)\rangle \leq \min_{\mathbb{S}^{d}}r,\quad \max_{\mathbb{S}^{d}}\langle x,n(x)\rangle = \max_{\mathbb{S}^{d}}r.\]
However, $\langle x,n(x)\rangle$ is $R_{0}\delta$-close to $R_{0}$ and hence $\Sigma$ lies within an annulus of size $R_{0}\delta$ around $\mathbb{S}^{d}(R_{0})$.
\end{proof}

\appendix

\section{Some basic inequalities}
\begin{lemma}\label{la:inequ1}
Fix $\Lambda > 1$, $\Upsilon > 0$, $\alpha > 0$, $\sigma > 0$.
Assume $\rho <  \brac{\frac{1}{2^\sigma \Lambda \Upsilon}}^{\frac{1}{\alpha}}$, then the following holds.

Let $f: [0,\rho] \to [0,\infty)$ be any nonnegative continuous function with $f(0) = 0$ and such that
\[
 f(r) \leq \Lambda (1+r^\alpha f(r))^\sigma\, \Upsilon \quad \forall r \in (0,\rho),
\]
then $\sup_{r \in [0,\rho]} f(r) \leq 2^\sigma \Lambda \Upsilon$.
\end{lemma}
\begin{proof}

We use the method of continuity.
Set
\[
 I := \{0\leq r\leq\rho: \quad f(r) \leq 2^\sigma \Lambda\, \Upsilon\}.
\]
Clearly $0  \in I$ since $f(0) = 0$.
If $r_k \in I$, $r = \lim_{k \to \infty} r_k$ then by continuity $f(r) = \lim_{k \to \infty} f(r_k) \leq 2^\sigma \Lambda \Upsilon$.
Lastly, if $r \in I$,
then we have
\[
\begin{split}
  f(r) \leq \Lambda (1+r^\alpha f(r))^\sigma \Upsilon \leq \Lambda (1+r^\alpha 2^\sigma \Lambda \, \Upsilon)^\sigma \Upsilon \\
  \end{split},
  \]
and since $r\leq \rho$ we find
\[
 f(r) \leq\Lambda \brac{1+1}^\sigma \Upsilon < 2^\sigma \Lambda \Upsilon.
\]
Thus, by continuity we have $f(\tilde{r}) < 2\Lambda \Upsilon$ for all $\tilde{r} \approx r$.

\end{proof}

\begin{lemma}[Morrey-Sobolev embedding]\label{la:morreysob}
Assume $s > \frac{d}{p}$. Then there exists $C=C(s,p,d)$ such that,
\begin{align}\begin{split}\nonumber
 &[Df]_{C^{s-\frac{d}{p}}(B(0,\tfrac 34 \rho))}
 \leq C \brac{\int_{B(0,\rho)}\brac{ \int_{B(0,\rho)} \frac{\abs{f(x)-f(y)-Df(y) (x-y)}}{|x-y|^{d+1+s}}\, dy}^p\, dx}^{\frac{1}{p}}.
 \end{split}
\end{align}

\end{lemma}
\begin{proof}
By scaling we may assume that $\rho = 1$.
Set
\[
 \Lambda := \brac{\int_{B(0,1)}\brac{ \int_{B(0,1)} \frac{\abs{f(x)-f(y)-Df(y) (x-y)}}{|x-y|^{d+1+s}}\, dy}^p\, dx}^{\frac{1}{p}}.
\]
By H\"older's inequality we observe that for any ball $B(x_0,r)$  we have
\[
\begin{split}
 \int_{B(x_0,r) \cap B(0,1)}\int_{B(x_0,r) \cap B(0,1)} \frac{\abs{f(x)-f(y)-Df(y) (x-y)}}{|x-y|^{d+1+s}}\, dy\, dx\\
 \leq |B(x_0,r) \cap B(0,1)|^{1-\frac{1}{p}} \Lambda \aleq_{d} r^{d(1-\frac{1}{p})}\, \Lambda.
\end{split}
 \]
There holds
\[|(D f(x)-D f(y))(x-y)| \leq 
|f(x)-f(y) -D f(y)(x-y)|+
|f(y)-f(x)-D f(x)(y-x)|
\]
 and hence by Fubini's theorem,
\[
 \int_{B(x_0,r) \cap B(0,1)}\int_{B(x_0,r)\cap B(0,1)} \frac{\abs{(D f(x)-D f(y)) (x-y)}}{|x-y|^{d+s+1}}\, dy\, dx \aleq_{d} r^{d(1-\frac{1}{p})}\, \Lambda.
\]
Thus
\[
 \int_{B(x_0,r) \cap B(0,1)}\int_{B(x_0,r) \cap B(0,1)} \abs{(Df(x)-Df(y)) \frac{x-y}{|x-y|}}\, dy\, dx \aleq_{d} r^{d+s} r^{d(1-\frac{1}{p})}\, \Lambda.
\]
Now consider
\begin{equation}\label{eq:bloodyhoelder:423}
\begin{split}
 |f(x)+f(y)-2f(\frac{x+y}{2})|
 &\leq \frac{1}{2}\int_0^1 |\langle \nabla f(\frac{x+y}{2} + t\frac{x-y}{2})-\nabla f(\frac{x+y}{2}- t\frac{x-y}{2}), x-y\rangle|\,dt\\
 &\leq \frac{1}{2}\int_0^1 |\langle \nabla f(\frac{1+t}{2} x + \frac{1-t}{2} y)-\nabla f(\frac{1-t}{2} x + \frac{1+t}{2} y), x-y\rangle|\,dt.\\
 \end{split}
\end{equation}
Now we observe that for $t \geq \frac{1}{2}$, substituting $\tilde{x} := \frac{1+t}{2} x + \frac{1-t}{2} y$, $\tilde{y} := \frac{1-t}{2} x + \frac{1+t}{2} y$ and observing that then $dx \aeq d\tilde{x}$, $dy \aeq d\tilde{y}$ and $|x-y| \aeq |\tilde{x}-\tilde{y}|$ (all since $t \geq \frac{1}{2}$) and using the convexity of $B(x_0,r) \cap B(0,1)$
\begin{equation}\label{eq:bloodyhoelder:424}
\begin{split}
&\int_{B(x_0,r) \cap B(0,1)}\int_{B(x_0,r) \cap B(0,1)} |\langle \nabla f(\frac{1+t}{2} x + \frac{1-t}{2} y)-\nabla f(\frac{1-t}{2} x + \frac{1+t}{2} y), x-y\rangle| \, dx\, dy\\
\aleq ~&r\int_{B(x_0,r) \cap B(0,1)}\int_{B(x_0,r) \cap B(0,1)} |\langle \nabla f(\tilde{x})-\nabla f(\tilde{y}), \frac{\tilde{x}-\tilde{y}}{|\tilde{x}-\tilde{y}|}\rangle| \, d\tilde x\, d\tilde y\\
\aleq ~&\Lambda\, r^{2d+1+s-\frac{d}{p}}.
\end{split}
\end{equation}
If on the other hand $t \leq \frac{1}{2}$ we decompose
\begin{equation}\label{eq:bloodyhoelder:425}
\begin{split}
&\int_{B(x_0,r) \cap B(0,1)}\int_{B(x_0,r) \cap B(0,1)} |\langle \nabla f(\frac{1+t}{2} x + \frac{1-t}{2} y)-\nabla f(\frac{1-t}{2} x + \frac{1+t}{2} y), x-y\rangle| \, dx\, dy\\
\leq&\int_{B(x_0,r) \cap B(0,1)}\int_{B(x_0,r) \cap B(0,1)} |\langle \nabla f(\frac{1+t}{2} x + \frac{1-t}{2} y)-\nabla f(x), x-y\rangle| \, dx\, dy\\
&+\int_{B(x_0,r) \cap B(0,1)}\int_{B(x_0,r) \cap B(0,1)} |\langle \nabla f(y)-\nabla f(\frac{1-t}{2} x + \frac{1+t}{2} y), x-y\rangle| \, dx\, dy\\
&+\int_{B(x_0,r) \cap B(0,1)}\int_{B(x_0,r) \cap B(0,1)} |\langle \nabla f(x)-\nabla f(y), x-y\rangle| \, dx\, dy.\\
\end{split}
\end{equation}
In the first term on the right-hand side of \eqref{eq:bloodyhoelder:425} we substitute $\tilde{y} := \frac{1+t}{2} x + \frac{1-t}{2}y$, observe that $\tilde{y}-x$ is parallel to $x-y$ and $|\tilde{y}-x| \aeq |x-y|$ and that $d\tilde{y} \aeq dy$ to obtain
\[
 \int_{B(x_0,r) \cap B(0,1)}\int_{B(x_0,r) \cap B(0,1)} |\langle \nabla f(\frac{1+t}{2} x + \frac{1-t}{2} y)-\nabla f(x), x-y\rangle| \, dx\, dy \aleq \Lambda\, r^{2d+1+s-\frac{d}{p}}.
\]
For the second term on the right-hand side of \eqref{eq:bloodyhoelder:425} we substitute $\tilde{x} := \frac{1-t}{2} x + \frac{1+t}{2} y$, observe that $\tilde{x}-y$ is parallel to $x-y$ with comparable norm, and thus 
\[
 \int_{B(x_0,r) \cap B(0,1)}\int_{B(x_0,r) \cap B(0,1)} |\langle \nabla f(y)-\nabla f(\frac{1-t}{2} x + \frac{1+t}{2} y), x-y\rangle| \, dx\, dy\aleq \Lambda\, r^{2d+1+s-\frac{d}{p}}.
\]
Thus, \eqref{eq:bloodyhoelder:425} becomes
\begin{equation}\label{eq:bloodyhoelder:426}
\begin{split}
&\int_{B(x_0,r) \cap B(0,1)}\int_{B(x_0,r) \cap B(0,1)} |\langle \nabla f(\frac{1+t}{2} x + \frac{1-t}{2} y)-\nabla f(\frac{1-t}{2} x + \frac{1+t}{2} y), x-y\rangle| \, dx\, dy\\
\aleq~&\Lambda\, r^{2d+1+s-\frac{d}{p}}.
\end{split}
\end{equation}
Using \eqref{eq:bloodyhoelder:423}, \eqref{eq:bloodyhoelder:424} and \eqref{eq:bloodyhoelder:425} we arrive at 
\[
 \int_{B(x_0,r) \cap B(0,1)} \int_{B(x_0,r) \cap B(0,1)}  |\frac{1}{2} f(x)+\frac{1}{2} f(y)-f(\frac{x+y}{2})|\, dy\, dx\leq C\, \Lambda r^{2d+1+s-\frac{d}{p}}.
 \]

For a.e. $x_0,y_0 \in B(0,1)$ we have 
\[
\begin{split}
&|f(x_0)+f(y_0)-f(\frac{x_0+y_0}{2})|\\ =& \lim_{k \to +\infty} \mvint_{B(0,2^{-k})} \mvint_{B(0,2^{-k})}   |f(x_0+z_1)+f(y_0+z_2)-2f(\frac{x_0+z_1+y_0+z_2}{2})|\, dz_1\, dz_2
\end{split}
\]
So if we set 
\[
a_k := \mvint_{B(0,2^{-k})} \mvint_{B(0,2^{-k})}   |f(x_0+z_1)+f(y_0+z_2)-2f(\frac{x_0+z_1+y_0+z_2}{2})|\, dz_1\, dz_2
\]
we have 
\[
|f(x_0)+f(y_0)-2f(\frac{x_0+y_0}{2})| = a_K + \sum_{k=K+1}^\infty (a_{k} - a_{k-1}) \aleq \sum_{k=K}^{\infty} a_k.
\]
But (observe the mean value!) when $x_0\in B(0,1)$ and $K$ is large enough, we have
\[
a_k \aleq \Lambda (2^{-k})^{1+s-\frac{d}{p}}
\]
so 
\[
\sum_{k=K}^\infty a_k \aleq \Lambda \brac{2^{-K}}^{1+s-\frac{d}{p}}.
\]
So if we choose $2^{-K} \aeq \min\{|x_0-y_0|,\dist(x_0,\partial B(0,1)),\dist(y_0,\partial B(0,1))\}$ then we have shown 
\[
|f(x_0)+f(y_0)-2f(\frac{x_0+y_0}{2})| \aleq \Lambda |x_0-y_0|^{1+s-\frac{d}{p}},
\]
whenever $|x_0-y_0| < \dist(y_0,\partial B(0,1)), \dist(x_0,\partial B(0,1))$.

In other words we have 
\[
|f(x_0+h)+f(x_0-h)-2f(x_0)| \aleq \Lambda |h|^{1+s-\frac{d}{p}} \quad \forall |h| < \dist(x_0,\partial B).
\]
From here one can show by elementary means that $Df$ is Hoelder continuous with exponent $s-d/p$.

\end{proof}


\end{document}